\documentclass[10pt,twoside,a4paper]{amsart}

\usepackage[english]{babel}
\usepackage[utf8]{inputenc}

\usepackage{amsfonts,amsthm,amssymb,amsmath,amsthm,latexsym,wasysym,stmaryrd,mathrsfs,dsfont,txfonts,xcolor,picins}
\usepackage{accents,multirow,rotating,subfigure,graphicx,bigstrut,lscape,multicol,enumerate,accents,mathtools,exscale}

\usepackage[normalem]{ulem}

\usepackage{pdftricks}
\usepackage{color}
\usepackage[pdftex,all]{xy}

\usepackage{hyperref}
\usepackage{appendix}
\usepackage[justification=centering]{caption}

\graphicspath{{Figures/}}

\numberwithin{equation}{section}

\theoremstyle{theorem}
\newtheorem {thmintro}{Theorem}
\newtheorem {corintro}{Corollary}
\newtheorem {thm}{Theorem}[section]
\newtheorem*{thm*}{Theorem}
\newtheorem {lem}[thm]{Lemma}
\newtheorem*{lem*}{Lemma}

\newtheorem*{prop*}{Proposition}
\newtheorem {cor}[thm]{Corollary}
\newtheorem*{cor*}{Corollary}

\newtheorem*{cor_proof*}{Corollary (of the proof)}

\newtheorem*{conjecture*}{Conjecture}
\theoremstyle{definition}
\newtheorem {defi}[thm]{Definition}
\newtheorem*{defi*}{Definition}

\newtheorem*{nota*}{Notation}
\theoremstyle{remark}
\newtheorem {rem}[thm]{Remark}
\newtheorem*{rem*}{Remark}

\newtheorem*{warning*}{Warning}

\newtheorem {convention}[thm]{Convention}
\newtheorem*{convention*}{Convention}

\newtheorem*{ex*}{Example}

\newtheorem*{question*}{Question}

\newtheorem*{questions*}{Questions}
\newtheorem {fact}[thm]{Fact}
\newtheorem*{fact*}{Fact}

\def\ov{\overline}
\def\1{\mathbf 1}

\def\C{{\mathcal C}}

\def\Z{{\mathds Z}}

\def\2Z{{\fract{\Z}{2\Z}}}

\newcommand{\fract}[2]{\hbox{\leavevmode 
\kern.1em \raise .25ex \hbox{\the\scriptfont0 $#1$}\kern-.1em }\big/
  {\hbox{\kern-.15em \lower .5ex \hbox{\the\scriptfont0 $#2$}} }}
\newcommand{\subfract}[2]{\hbox{\leavevmode
  \kern0em \raise .25ex \hbox{\the\scriptfont0 \tiny $#1$}\kern-.1em }/
  {\hbox{\kern-.15em \lower .5ex \hbox{\the\scriptfont0 \tiny $#2$}} }}

\newcommand{\dessin}[2]{
  \vcenter{\hbox{\includegraphics[height=#1]{#2.pdf}}}}

\definecolor{pink}{rgb}{0.858, 0.188, 0.478}
\definecolor{orange}{rgb}{1, 0.647, 0}
\definecolor{vert}{rgb}{0.42, 0.557, 0.137}

\newcommand{\AI}{\hspace{0pt}\kern0pt}

\begin{document} 

\title{Higher order unknotting numbers for knots and links} 
\author[J.B. Meilhan]{Jean-Baptiste Meilhan} 
\address{Univ. Grenoble Alpes, CNRS, Institut Fourier, F-38000 Grenoble, France}
	 \email{jean-baptiste.meilhan@univ-grenoble-alpes.fr}
\author[H. Ogita]{Harumi Ogita}
         \address{Waseda University, 1-6-1 Nishi-Waseda,
  Shinjuku-ku, Tokyo 169-8050, Japan}
         \email{harumi@aoni.waseda.jp}
\author[A. Yasuhara]{Akira Yasuhara} 
\address{Faculty of Commerce, Waseda University, 1-6-1 Nishi-Waseda,
  Shinjuku-ku, Tokyo 169-8050, Japan}
	 \email{yasuhara@waseda.jp}
%
%
\begin{abstract} 
Ohyama, Taniyama and Yamada showed that finite type knot invariants are insensitive to the unknotting number: for any integer $n$ and any knot $K$, there is an unknotting number $1$ knot which cannot be distinguished from $K$ by any finite type invariant of order $\le n$. We generalize this result in three directions. 
First, we prove an analogous result for all higher order unknotting number given by $C_n$-moves. 
Second, we give several extensions to the link case. 
Third, we show that a similar result holds for ribbon $2$-knots.
\end{abstract} 

\maketitle

\section{Introduction}
The unknotting number, the minimal number of crossing changes needed to trivialize a knot, is one of the oldest invariant 
of knot theory. It remains however quite mysterious and notoriously difficult to compute; it was only very recently that its non-additivity under connected sum was established \cite{BH26}. 

A natural question is whether other, more computable knot invariants can help estimate the unknotting number. 
The question was in particular raised for finite type (or Vassiliev-Goussarov) invariants, which form a rich class of knot invariants. 
A result of Ohyama-Taniyama-Yamada, subsequently reproved by Ohyama, answers by the negative: 
\begin{thm}\label{thm:OTY} \cite{OTY,O}
Let $K$ be a knot and $n$ be an integer ($n\ge 2$). 
There exists an unknotting number $1$ knot $J$ such that $v_n(K)=v_n(J)$ for all finite type invariants $v_n$ of degree $\le n$. 
\end{thm}
\noindent In other words, finite type knot invariants are insensitive to the unknotting number. 

Another classical unknotting operation is the $\Delta$-move \cite{MN,Matveev}, illustrated on the left-hand side of Figure \ref{fig:delta}. 
\begin{figure}[!h]
  \includegraphics[scale=0.7]{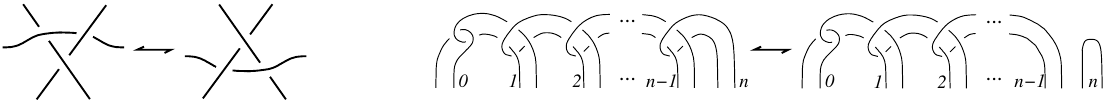} 
  \caption{A $\Delta$-move (left) and a $C_n$-move (right).} {\footnotesize A $C_2$-move is equivalent to a $\Delta$-move, while a $C_1$-move is equivalent to a crossing change}  \label{fig:delta}
\end{figure}
The $\Delta$-unknotting number $u_\Delta$ has also been the subject of intensive studies.
Okada proved that a lower bound is given by the absolute value of the Casson knot invariant $a_2$, that is, the second coefficient of Conway polynomial \cite{okada}, which is up to scalar the unique finite type knot invariant of degree $2$. However, the second author proved that finite type invariants provide no further information on the  $\Delta$-unknotting number: 
\begin{thm}\label{thm:harumi} \cite{Harumi}
Let $K$ be a knot such that $a_2(K)=0$, and let $n$ be any integer ($n\ge 3$). 
There exists a knot $J$ with $\Delta$-unknotting number $2$ such $v_n(K)=v_n(J)$ for all finite type invariants $v_n$ of degree $\le n$. 
\end{thm}
\medskip

The purpose of the present paper is to provide several generalizations of these results, outlined in the next three subsections. 

\subsection{Higher order unknotting numbers for knots}\label{sec:introknots}

The crossing change and $\Delta$-move can be seen as the first two of an infinite family of local moves called \emph{$C_n$-moves},  illustrated in the right-hand side of Figure \ref{fig:delta}. 
These local moves were introduced by Habiro \cite{Hmt}, and can hence be seen as higher order crossing changes. 
For each $n$, the $C_n$-move generates an equivalence relation called \emph{$C_n$-equivalence}; two knots are $C_n$-equivalent if and only if they  cannot be distinguished by any finite type invariant of degree $<n$ \cite{Goussarov,Habiro}. 

A knot $K$ is called \emph{$C_n$-trivial} if it is $C_n$-equivalent to the unknot, that is, if it can be trivialized by a sequence of $C_n$-moves: we denote by $u_n(K)$, and call \emph{$C_n$-unknotting number} of $K$, the minimal number of $C_n$-moves in such a sequence. Note that $u_1$ is the usual unknotting number. 
Since finite type invariants of degree $<n$ give lower bounds for the $C_n$-unknotting number \cite{OY,Tsukamoto,Kanenobu}, 
one naturally wonders whether degree $\ge n$ invariants provide further information. 
The following, our first main result, answers by the negative. 
\begin{thmintro}\label{thm:knots}
Let $K$ be knot with vanishing finite type invariants of degree $< k+1$. 
For all $n\ge k+1$, there exists a knot $J$ such that $u_k(J)\le k!$ and $v_n(K)=v_n(J)$  for all finite type invariants $v_n$ of degree $\le n$. 
\end{thmintro}
\noindent The case $k=2$ is precisely Theorem \ref{thm:harumi}, since a $\Delta$-move is equivalent to a $C_2$-move (note that the $C_2$-triviality assumption is vacuous). We provide here an alternative proof based on Habiro's clasper theory \cite{Habiro}, which applies naturally at all higher orders. 
See Remark \ref{rem:k=1} for the case $k=1$. 

\begin{rem}
 Using the same argument as in the proof of Theorem \ref{thm:knots}, given in Section \ref{sec:knot}, we can actually prove the following stronger version. Let $K$ and $M$ be two knots that cannot be distinguished by any finite type invariants of degree $< k+1$. For all $n\ge k+1$,  there exists a knot $J$ which can be deformed into $M$ by at most $k!$ $C_k$-moves, and such that $v_n(K)=v_n(J)$  for all finite type invariants $v_n$ of degree $\le n$.  This was already the case in \cite{Harumi}, where Theorem \ref{thm:harumi} is also stated in this stronger form. Similar remarks hold for Theorems \ref{thm:linklh},\footnote{More precisely, for Theorem \ref{thm:linklh}, a similar remark holds for string links rather than links.} \ref{thm:linkconc} and \ref{thm:2knots}, but we will confine ourselves to statements (corresponding to the case where $M$ is trivial) focussing on higher order unknotting numbers.
\end{rem} 

\subsection{Higher order unknotting numbers for links}\label{sec:introlinks}

We next consider upper bounds for the higher order unknotting numbers of links, under additional topological conditions. 
\medskip

On the one hand, we consider the \emph{link-homotopy} relation, due to Milnor \cite{Milnor}, generated by self-crossing changes, that is, crossing changes among two strands of a same link component. The following is proved in Section \ref{sec:prooflh}.
\begin{thmintro}\label{thm:linklh}
Let $L$ be an $n$-component link with vanishing Milnor link-homotopy invariants of length $\le k+1$ ($k\ge 2$). 
There exists an $n$-component link $J$ such that $u_k(J)\le \binom{n-1}{k+1}k!$ and $J$ is link-homotopic to $L$. 
\end{thmintro}
\begin{rem}\label{rem:introk=1}
As outlined in Remark \ref{rem:k=1link}, the proof of Theorem \ref{thm:linklh} also applies to the case $k=1$ with minor adaptation, proving the following: 
If $L$ is an $n$-component  algebraically split link, there exists a link $J$ such that $u_1(J)\le (n-1)(n-2)$ and which is link-homotopic to $L$. 
Recall that a link is called algebraically split if all pairwise linking numbers vanish, that is, if all Milnor invariants of length $\le 2$ vanish. By \cite{MN}, a link is algebraically split if and only if it is $C_2$-trivial. 
\end{rem}

We stress that a link can be trivialized by a sequence of $C_k$-moves and link-homotopies if and only if it has vanishing Milnor link-homotopy invariants of length $\le k$, see Lemma \ref{rem:Cm+lh}.  
By Lemma \ref{lem:lhclasper} and \cite[Prop.~3.22]{Habiro}, Theorem \ref{thm:linklh} reformulates as follows. 
\begin{corintro}
Let $L$ be an $n$-component link with vanishing Milnor link-homotopy invariants of length $\le k+1$ ($k\ge 2$). The minimal number $n_k(L)$ of $C_k$-moves needed to transform $L$ into a link-homotopically trivial link satisfies $n_k(L)\le  \binom{n-1}{k+1}k!$. 
\end{corintro}
\begin{rem}
The case $k=1$ is given by Bosman, Davis, Martin, Otto and Vance in \cite[Cor.~1.3]{BDMOV}. While their argument is algebraic, we recover this result in a geometric way (see Remark \ref{rem:k=1link}) . 
The case $k=2$ says that if $L$ has vanishing linking numbers and triple linking numbers, then $n_2(L)\le \frac{1}{3}(n^3-6n^2+11n-6)$. 
This improves significantly the upper bound $\frac{2}{3}(n^3-3n^2+2n-6)$ given in \cite[Cor.~1.6]{BDMOV}. The higher order cases do not seem to have been previously investigated. 
\end{rem}
\medskip 

On the other hand, we consider links up to concordance, and obtain the following bound for the $\Delta$-unknotting number for string links. 
\begin{thmintro}\label{thm:linkconc}
Let $L$ be a $C_3$-trivial $n$-component string link and let $p$ be an integer. 
There exists a string link $J$ with $\Delta$-unknotting number $u_2(J)\le \frac{1}{3}(n^3-n)$  
such that $v_p(L)=v_p(J)$ for all finite type concordance invariants $v_p$ of degree $\le p$. 
\end{thmintro}
Recall that Milnor invariants of length $\le k$ are finite type concordance invariants of degree $<k$ for string links. 
Nonetheless the following, which is a consequence of the proof of Theorem \ref{thm:linkconc}, holds for links. 
\begin{corintro}\label{cor:linkconc}
Let $L$ be a $C_3$-trivial $n$-component link and let $p$ be an integer. 
The minimal number of $\Delta$-moves needed to transform $L$ into a link with vanishing Milnor invariants of length $\le p$ is at most $\frac{1}{3}(n^3-n)$. 
\end{corintro}

The proofs of Theorem \ref{thm:linkconc} and Corollary \ref{cor:linkconc} are given in Section \ref{sec:proofconc}. We discuss several extensions of Theorem \ref{thm:linkconc} in Section \ref{sec:moreconc}.

\subsection{Unknotting number for ribbon $2$-knots}\label{sec:introribbon}

Finally, we consider the higher-dimensional version of Theorem \ref{thm:OTY}, that is, for $2$-knots, which are smooth embeddings of  $S^2$ in $S^4$. A major difference with classical knot theory is that there is no notion of crossing change for $2$-knots in general. 
Such a notion however exists for \emph{ribbon $2$-knots}, which are obtained by adding $m$ $1$-handles to a trivial $2$-link with $m + 1$ components, for some $m$.   
Any ribbon $2$-knot admits a (broken surface) diagram where all singularities are circles of double points as in one of the two local pictures shown in Figure \ref{fig:crossing}. We call \emph{ribbon crossing change} the local move on ribbon $2$-knots that exchanges one of these two configurations by the other \cite{HKS}.  
 \begin{figure}[!h]
  \includegraphics[scale=0.5]{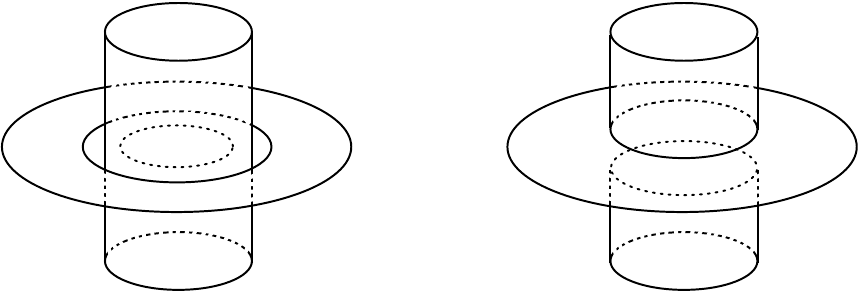}
  \caption{The two types of crossings in a ribbon $2$-knot diagram.}\label{fig:crossing}
 \end{figure}
Any ribbon $2$-knot can be deformed into the trivial one by a sequence of ribbon crossing changes, so that we can define the \emph{ribbon unknotting number} as the minimal number of changes in such a sequence. 

There is a theory of finite type invariants for ribbon $2$-knots based on the ribbon crossing change, developed in \cite{HKS,HS}. It is shown in \cite{HS} (see also \cite{arrow}) that any such finite type invariants is a combination of the coefficients $\alpha_k$ of the normalized Alexander polynomial of $2$-knots (see Section \ref{sec:alex}).
\begin{thmintro}\label{thm:2knots}
Let $R$ be a ribbon $2$-knot and $n$ be an integer ($n\ge 2$). 
There exists a ribbon $2$-knot $J$ with ribbon unknotting number $1$ such that 
 $\alpha_k(R)=\alpha_k(J)$ for all $k\le n$. 
\end{thmintro}
The proof, given in Section \ref{sec:proofribbon}, makes use of the theory of \emph{welded knotted objects}, see Section \ref{sec:Tube}. 

Although analogues of Theorem \ref{thm:knots} and \ref{thm:linklh} can also be given for ribbon $2$-knots and $2$-links, we will not further develop this direction in this paper. 
\medskip

\noindent \textbf{AI Statement.}
This work was realized without any use of generative AI. 

\section{Preliminaries: a clasper calculus toolbox}\label{sec:prelim}

The proofs of Theorems \ref{thm:knots}, \ref{thm:linklh} and \ref{thm:linkconc} make centrally use of Habiro's clasper theory \cite{Habiro}.
In this section we review all necessary definitions and results required for these proofs. 

Recall that a string link is a pure tangle without closed components in $D^2\times [0,1]$ (see \cite{HL} for a precise definition).  
For a fixed number of components, string links admit a natural composition given by the stacking product, denoted by $\cdot$, and with the trivial string link $\1$ as unit element. 
In what follows we simply denote by $I$ the unit interval $[0,1]$. 

\subsection{Claspers and $C_k$-equivalence}

Let $L$ be a string link in $D^2\times I$.
\begin{defi}
A \emph{clasper for $L$} is an embedded (possibly disconnected) surface $C$ in $D^2\times I$ such that:
\begin{enumerate}
\item[(i)] $C$ decomposes into disks and bands, called \emph{edges}, each connecting two distinct disks;
\item[(ii)]  disks have either 1 or 3 incident edges, and are called \emph{leaves} or \emph{nodes} respectively;
\item[(iii)] $C$ intersects $L$ transversely, and the intersection is contained in the union of the interior of the leaves. 
\end{enumerate}
A clasper for $L$ is \emph{simple} if each of its leaves intersects $L$ at one point.
\\
The degree of a connected clasper is defined as half the number of nodes and leaves.
\end{defi}
\begin{rem}
The definition of claspers given here is somewhat more restrictive that the one given in \cite{Habiro}. In particular, we do not consider \lq boxes' here (see however Appendix \ref{sec:A}); also, what we call here \lq leaves' actually correspond to \lq disk-leaves' in \cite{Habiro}. 
\end{rem}

\begin{defi}
A clasper which is a disk is called a \emph{tree clasper}. 
A degree $k$ tree clasper is called a \emph{$C_k$-tree}. 
\end{defi} 
We will use the usual drawing convention for claspers given in \cite[Fig.~7]{Habiro}, except for the following: a $\oplus$ (resp. $\ominus$) on an edge denotes a positive (resp. negative) half-twist.
\medskip

Given a clasper $C$ for $L$, there is a procedure to construct a framed link $L(C)$ in a regular neighborhood of $C$, see Figures 1 and 6 in \cite{Habiro}. 
The result of surgery on $D^2\times I$ along $L(C)$ is canonically diffeomorphic to $D^2\times I$, so that this surgery operation can be regarded as a local move on $L$ in $D^2\times I$: 
we say that the resulting string link $L_C$ in $D^2\times I$ is obtained by \emph{surgery on $L$ along the clasper $C$}.  
For example, surgery along a simple $C_k$-tree as illustrated in Figure \ref{ckmove} is equivalent to band-summing a copy of the $(k+1)$-component Milnor link (see \cite[Fig. 7]{Milnor}), and is equivalent to a $C_k$-move as defined in Figure \ref{fig:delta}.  
\begin{figure}[!h]
\includegraphics{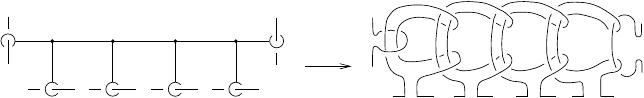}
\caption{Surgery along a linear $C_5$-tree.} \label{ckmove}
\end{figure}
\begin{defi}\label{def:linear}
A simple $C_k$-tree having the shape of the tree clasper in Figure \ref{ckmove} is called \emph{linear}, 
and the left-most and right-most leaves in the figure \ref{ckmove} are called the \emph{ends}.  
\end{defi}

The \emph{$C_k$-equivalence} as defined in the introduction coincides with the equivalence relation generated by surgery along degree $k$ claspers and isotopies; we denote it by  $\stackrel{C_{k}}{\sim}$. 
Habiro proved that the $C_k$-equivalence relation becomes finer as $k$ increases, 
and that two (string) links are $C_k$-equivalent if and only if they are related by surgery along
simple $C_k$-trees \cite{Habiro}. 
\medskip

Recall that a leaf of a clasper $C$ for a (string) link $L$ is a \emph{trivial leaf} if its
boundary bounds a disk whose interior is disjoint from $L\cup C$. 
\begin{lem}[Trivial leaf]\label{lem:trivial}
Let $C$ be a connected clasper for a (string) link  $L$. If $C$ contains a trivial leaf, then $L_C$ is equivalent to $L$.
\end{lem}
The following terminology will be useful in the next sections.
\begin{defi}\label{def:Btrivial}
Let $B\cup C$ be a clasper for a (string) link  $L$. We say that the clasper $C$ is \emph{$B$-trivial} if each component of $C$ has at least one trivial leaf when deleting $B$.   
\end{defi}
See the right-hand side of Figure \ref{fig:tech} for an example, where the clasper $C$ is $T$-trivial. 

\subsection{Calculus of claspers}

We next collect several classical results proved in \cite[\S 4]{Habiro}, which allow to manipulate and \lq simplify' claspers up to higher degree ones. In all statements, $\1$ still denotes the trivial string link, for some given number of components. 
\begin{lem} \label{lem:calculus}
Let $F$ be a clasper for a (string) link $L$, a let $C$ be a degree $k$ clasper in $F$.
Let $F'=(F\setminus C)\cup C'$, where $C'$ is obtained from $C$ by either of the following operations:\\[-0.1cm]
\begin{minipage}{9.5cm}
\begin{itemize}
\item passing an edge of $C$ across $L$;
\item passing an edge of $C$ across an edge of some clasper in $F$ (possibly $C$ itself);  
\item sliding a leaf of $C$ over a leaf of another component of $F$, as shown on the right. 
\end{itemize}
\end{minipage}
\begin{minipage}{3.5cm}
\,\,\,\,\,\, \includegraphics[scale=1.5]{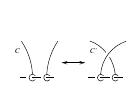}\\[-0.45cm]
\end{minipage}
We have $$L_F \stackrel{C_{k+1}}{\sim} L_{F'}.$$  
\end{lem}
 
\begin{lem}[Twist] \label{lem:twist}
 Let $C$ be a degree $k$ clasper for $\1$ and let $C^{-1}$ be obtained from $C$ by adding a half-twist on some edge. We have  
 $$\1_C\cdot \1_{\C^{-1}} \stackrel{C_{k+1}}{\sim} \1.$$  
As a consequence, one can freely move a half-twist among the edges of a $C_k$-tree up to $C_{k+1}$-equivalence. 
 \end{lem}

The next lemma is a clasper versions of the AS, IHX and STU relations of Jacobi diagrams. 
\begin{lem}[AS, IHX and STU] \label{lem:asihx}
Let $T$ and $T'$ (resp. $T_I$, $T_H$ and $T_X$, resp. $G_S$, $G_T$ and $G_U$)  be  degree $k$ claspers for $\1$ which differ only in a small ball  as depicted in Figure \ref{fig:clasper_calc}. 
 \begin{figure}[!h]
  \includegraphics[scale=1]{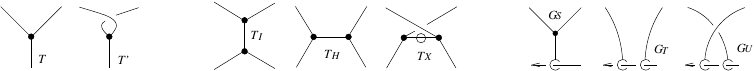}
  \caption{AS (left), IHX (center) and STU (right) relations for claspers.}\label{fig:clasper_calc}
 \end{figure}
We have 
$$\textrm{$\1_{T}\cdot \1_{T'} \stackrel{C_{k+1}}{\sim}\1$, $\quad \1_{T_I} \stackrel{C_{k+1}}{\sim} \1_{T_H}\cdot \1_{T_X}\quad $ and $\quad  \1_{G_S}\cdot \1_{G_T} \stackrel{C_{k+1}}{\sim} \1_{G_U}$.} $$ 
\end{lem}
We will make use of the following standard consequence of the AS and IHX relations. 
\begin{cor}\label{cor:asihx}
 Let $T$ be a simple $C_k$-tree for $\1$. Pick two leaves of  $T$, intersecting the $i$th and $j$th components of $\1$ (possibly $i=j$).  Then $\1_T\stackrel{C_{k+1}}{\sim} \prod_k \1_{L_k}$, where each $L_k$  is a linear $C_k$-tree whose ends intersect the $i$th and $j$th components of $\1$. 
\end{cor}

We now state a technical result which shall play a pivotal role in the proofs of our main results. 
The proof requires several extra ingredients from clasper theory, and is postponed to Appendix \ref{sec:A}.
\begin{lem}\label{lem:tech}
Let $T\cup P$ be a union of a $C_k$-tree and a $C_m$-tree for $\1$ ($m\ge k+1\ge 2$) which is as shown on the left-hand side of Figure \ref{fig:tech} in a neighborhood of $T$. Then $$1_{T\cup P} \stackrel{C_{m+1}}{\sim} \1_{T\cup C},$$
where $C$ is the $T$-trivial  $C_{m-k}$-tree shown in the right-hand side of the figure. 
 \begin{figure}[!h]
  \includegraphics[scale=0.9]{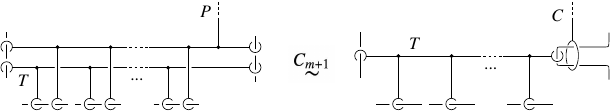}
  \caption{Lemma \ref{lem:tech}}\label{fig:tech}
 \end{figure}
\end{lem}
\begin{rem}\label{rem:tech}
 We also need the following mild variation of Lemma \ref{lem:tech}, whose proof is strictly similar. 
 Let $T\cup P$ be a union of a $C_k$-tree and a $C_m$-tree for $\1$ ($m\ge k+1\ge 3$) which is as shown on the left-hand side of Figure \ref{fig:techbis} in a neighborhood of $T$. Then $1_{T\cup P} \stackrel{C_{m+1}}{\sim} \1_{T\cup C}$, 
where $C$ is the $T$-trivial   $C_{m-k}$-tree shown in the right-hand side of Figure \ref{fig:techbis}. 
  \begin{figure}[!h]
  \includegraphics[scale=0.9]{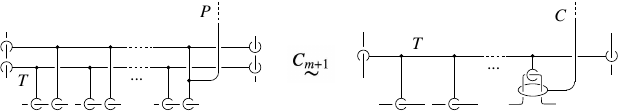}
  \caption{A variant of Lemma \ref{lem:tech}}\label{fig:techbis}
 \end{figure}
\end{rem}

\subsection{Claspers and link-homotopy}\label{sec:ck+h}

Recall that link-homotopy is the equivalence relation on links generated by self-crossing changes \cite{Milnor}. 
\begin{lem}\label{lem:lhclasper} \cite{FY}
Two (string) links are link-homotopic if and only if they are related by surgery along \emph{repeated} claspers. 
\end{lem}
Here a clasper is called \emph{repeated} if it intersects some (string) link component at two distinct leaves (note that a self-crossing change is surgery along a repeated $C_1$-tree). 
\medskip

We next introduce the following notion. 
\begin{defi}
 The \emph{$C_k$-homotopy } ($k\ge 1$) is the equivalence relation on (string) links, denoted by $\stackrel{C_k+h}{\sim}$, generated by $C_k$-equivalence and link-homotopy. 
\end{defi}

Recall that \emph{Milnor invariants} are extracted from the quotients of the fundamental group of the (string) link complement by the successive terms of its lower central series \cite{Milnor,Milnor2}. They are indexed by a sequence of indices in $\{1,\ldots,n\}$, where $n$ is the number of components, and the \emph{length} of each Milnor invariant is simply the length of the corresponding sequence. 
Casson observed that all Milnor invariants are concordance invariants \cite{Casson}, and Milnor showed that Milnor invariants indexed by \emph{non-repeated} sequences are even link-homotopy invariants \cite{Milnor}: one often refers to the latter as \emph{Milnor link-homotopy invariants}. Habiro proved that length $\le k$ Milnor invariants are $C_k$-equivalence invariants. 
See \cite{survey} for further reading.  
\begin{lem}\label{rem:Cm+lh}
A (string) link $L$ has vanishing Milnor link-homotopy invariants of length up to $k$ ($k\ge 2$) if and only if $L$ is $C_{k}$-homotopic to an unlink.
\end{lem}
\begin{proof}
 The if part is clear, since length $\le k$ Milnor link-homotopy invariants are invariants under both $C_k$-equivalence and link-homotopy. The only if part is a consequence of the clasper-based proof of the link-homotopy classification of string links given in \cite[\S~4]{yasuTAMS} by the last author, as follows. By \cite[Thm.~4.3]{yasuTAMS}, an $n$-component string link $L$ is link-homotopic to a product $\prod_{j=1}^{n-1} L_j$, where each $L_j$ is obtained from $\1$ by surgery along a union of simple linear $C_j$-trees and is completely determined by Milnor link-homotopy invariants of $L$ of length $\le j+1$. Hence $L$ is $C_{k}$-homotopic to the product $\prod_{j=1}^{k-1} L_j$, which is completely determined by  Milnor link-homotopy invariants of $L$ of length $\le k$. This implies that string links are classified up to $C_k$-homotopy by  Milnor link-homotopy invariants of length $\le k$. 
 \end{proof}

\subsection{Claspers and concordance}\label{sec:ck+c}

We next turn to the concordance relation for (string) links.

We say that a clasper is \emph{looped} is none of its connected component is a tree clasper.
\begin{lem}(\cite{CT,GL})\label{lem:loop} 
Let $C$ be a looped clasper for a (string) link $L$.  Then $L_C$ is concordant to $L$.  
\end{lem}
\begin{lem}\label{cor:loop}
Let $L$ be a (string) link, and let $T$ be a $C_{k}$-tree for $L$ ($k\ge 2$). \\[-0.35cm]

\noindent \begin{minipage}{10.8cm}
 \begin{enumerate}
  \item Suppose that $L\cup T$ intersects a $3$-ball as illustrated on the right. 
        Then $L_T\stackrel{C_{k+1}}{\sim} L_G$ for some connected looped clasper $G$ of degree $k$. 
  \item Let $T'$ be obtained by sliding a leaf of $T$ across another leaf of $T$, as in the figure of Lemma \ref{lem:calculus}. Then $L_{T'}\stackrel{C_{k+1}}{\sim} L_{T\cup G} $ for some connected looped clasper $G$ of degree $k$. 
 \end{enumerate}
\end{minipage}
\begin{minipage}{3.5cm}
\,\,\,\,\,\, \includegraphics[scale=1.5]{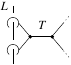}\\[0.5cm]
\end{minipage}
\end{lem}
\begin{proof}
 Assertion (1) follows from \cite[Lem.~2.8]{MYAGT}. Assertion (2) follows from the STU relation (Lemma \ref{lem:asihx}).
 \end{proof}

We shall also need the following notion from \cite{MY5}.
\begin{defi}
 The \emph{$C_k$-concordance} ($k\ge 1$) is the equivalence relation on (string) links, denoted by $\stackrel{C_k+c}{\sim}$, generated by $C_k$-equivalence and concordance.
\end{defi}
This notion is closely related to the order $k$ \emph{Whitney concordance} studied by Conant, Schneiderman and Teichner in \cite{CST}; it is indeed shown in \cite{CST2} that a link is $C_k$-concordant to an unlink if and only if it bounds a degree $k$ Whitney tower in the $4$-ball. 

\section{The knot case: proof of Theorem \ref{thm:knots}}\label{sec:knot}

For convenience, in this proof we implicitly regard $K$ as a long knot, i.e. as a $1$-component string link. 
Recall that the closure map from long knots to (usual) knots is bijective. Denote by $\1$ the trivial long knot. 

Let $\widetilde{S}_k$ be the subgroup of the symmetric group $S_k$ formed by permutations $\sigma$ such that $\sigma(k-1)<\sigma(k)$. 
For any $\sigma\in \widetilde{S}_k$, let $B_\sigma$ and $B^{-1}_\sigma$ be the $C_k$-tree for $\1$ represented on the left-hand side of Figure \ref{fig:Tsigma}, using Convention \ref{conv} below.
 \begin{figure}[!h]
  \includegraphics[scale=1]{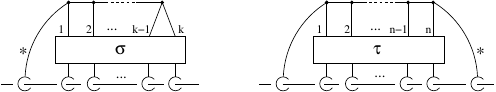}
  \caption{The $C_k$-tree $B_\sigma$ (left) and the $C_{n+1}$-tree $T_\tau$ (right). }\label{fig:Tsigma}
 \end{figure} 
\begin{convention}\label{conv}
 In all figures defining a clasper $C$, we implicitly also define the clasper $C^{-1}$ obtained by inserting a negative half-twist on the $\ast$-marked edge. Also, a rectangle labeled by a permutation represents the unique positive braid realizing this permutation.
\end{convention}

The notation $B^{-1}_\sigma$ is justified by Lemma \ref{lem:twist}, which tells us that 
\begin{equation}\label{eq:0}
\textrm{For all $\sigma\in \widetilde{S}_k$,\,\,\,\,\,\,\,} \1_{B_\sigma}\cdot \1_{B^{-1}_\sigma} \stackrel{C_{k+1}}{\sim} \1.
\end{equation} 

Fix an arbitrary total order on $\widetilde{S}_k$. We denote by $B$ the split union of $C_k$-trees $\sqcup_{\sigma\in \widetilde{S}_k} (B_\sigma\sqcup B^{-1}_\sigma)$, met successively along $\1$ in this fixed order when following the orientation. 
Observe that $B$ consists of exactly $k!$ $C_k$-trees. 

We shall prove by induction that for all $n\ge k$ we have 
$$ K\stackrel{C_{n+1}}{\sim} \1_{B\cup E_n}, $$
where $E_n$ is a union of $B$-trivial tree claspers for $\1$ (Definition \ref{def:Btrivial}). 
Since $B$ is a union of $k!$ $C_k$-trees, and since $E_n$ is $B$-trivial, the (long) knot $J_n:=\1_{B\cup E_n}$ clearly satisfies $u_k(J_n)\le k!$, and since finite type invariants of degree $\le n$ are $C_{n+1}$-equivalence invariants, this will imply the desired statement.  

The base case of the induction is clear, since by Equation (\ref{eq:0}) we have  $ K\stackrel{C_{k+1}}{\sim} \1_{B}$. 
Now assume that $ K\stackrel{C_{n+1}}{\sim} \1_{B\cup E_n}$ for some union $E_n$ of $B$-trivial tree claspers. 
Then by \cite[Prop.~3.22]{Habiro} there exists a union $F_{n+1}$ of simple $C_{n+1}$-trees for $\1$ such that 
$$ K = \1_{B\cup E_n\cup F_{n+1}}. $$
We stress that the claspers in $F_{n+1}$ may be knotted and linked with each other and with $B\cup E_n$ in an arbitrary way. 
However, by clasper calculus we can separate and simplify those $C_{n+1}$-trees up to $C_{n+2}$-equivalence; more precisely, by Lemma \ref{lem:calculus} we have  
 $$ K \stackrel{C_{n+2}}{\sim} \1_{B\cup E_n}\cdot P_{n+1}, $$
where $P_{n+1}$ is a product, each factor being obtained from $\1$ by surgery along a single $C_{n+1}$-tree. 
Furthermore, by Corollary \ref{cor:asihx} and Lemma \ref{lem:twist}, we may assume that each of these $C_{n+1}$-tree in $P_{n+1}$ is a copy of the clasper $T_\tau$ or $T^{- 1}_\tau$ shown on the right-hand of Figure \ref{fig:Tsigma} for some $\tau\in S_{n}$. 
The core idea is that, up to $C_{n+2}$-equivalence, each such factor can be eliminated using a suitable $C_k$-tree in $B$. 
More precisely, suppose that $P_{n+1}$ contains a factor $\1_{T_\tau}$ for some $\tau\in S_n$ (the case involving $T^{- 1}_{\tau}$ is strictly similar, and is omitted). Consider the restriction $\overline{\tau}\in S_k$ of $\tau$ to $\{1,\ldots,k\}$. 
If $\overline{\tau}\in \widetilde{S}_k$, then by Lemma \ref{lem:calculus} we can deform $T_\tau$ up to $C_{n+2}$-equivalence so that, in a neighborhood of $B_{\overline{\tau}}$, the union $B_{\overline{\tau}}\cup T_\tau$ is precisely as shown on the left-hand side of Figure \ref{fig:tech}.
Using Lemma \ref{lem:tech}, we obtain that $\1_{ B_{\overline{\tau}}\cup T_\tau}\stackrel{C_{n+2}}{\sim} \1_{B_{\overline{\tau}}\cup C}$, where $C$ is a $B$-trivial $C_{n-k}$-tree (see the right-hand side of Figure \ref{fig:tech}).  
If $\overline{\tau}\in S_k\setminus \widetilde{S}_k$, then observe that $\tau'=(k-1\,k)\circ \overline{\tau}\in \widetilde{S}_k$, where $(k-1\,k)$ is the  transposition exchanging $k-1$ and $k$. Applying first the AS Lemma \ref{lem:asihx}, then the Twist Lemma \ref{lem:twist}, we can 
use Lemma \ref{lem:calculus}  to deform $T_\tau$ up to $C_{n+2}$-equivalence so that $B_{\tau'}\cup T_\tau$ is as shown on the left-hand side of Figure \ref{fig:techbis}. Remark \ref{lem:tech} then shows that $\1_{ B_{\tau'}\cup T_\tau}\stackrel{C_{n+2}}{\sim} \1_{B_{\tau'}\cup C'}$, where $C$ is a $B$-trivial $C_{n-k}$-tree.  

Applying this procedure successively to all factors in $P_{n+1}$, yields the desired equivalence. 

\begin{rem}\label{rem:k=1}
The proof of Theorem  \ref{thm:knots} also applies to the case $k=1$, that is, to recover Theorem \ref{thm:OTY}.
The argument is strictly the same, once an appropriate substitute for the clasper $B$ is chosen. 
For the $k=1$ case, we simply chose $B$ to be a single $C_1$-tree for $\1$ shown below. 
\[ \textrm{\includegraphics[scale=1.3]{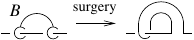}} \]
As the figure illustrates, the result of surgery on $\1$ along $B$ is equivalent to $\1$ itself. We can then prove, by exactly the same induction as above and using Lemma \ref{lem:tech}, that for all $n\ge k$ we have 
$ K\stackrel{C_{n+1}}{\sim} \1_{B\cup E_n}$ for some union $E_n$ of $B$-trivial tree claspers. 
The fact that $B$ consists here of a single clasper provides the desired unknotting number $1$ property.
\end{rem}

\section{The link case: proof of Theorems \ref{thm:linklh} and \ref{thm:linkconc}.}  \label{sec:link}

\subsection{$C_k$-unknotting number and link-homotopy: proof of Theorem \ref{thm:linklh}}\label{sec:prooflh}

We can freely prove the result for string links rather than links. Indeed, a link has vanishing Milnor invariants up to length $\le k$ if and only if it is the closure of a string link with same vanishing property, and a similar equivalence holds for link-homotopy. 
We denote by $\1$ the $n$-component trivial string link. 

The strategy is similar to the one used for the knot case in the previous section.

We will need the following general definition. 
Let $s=i_0 i_1 \ldots i_{k}$ be a subsequence of $12...n-1$, for some $k< n$; 
note in particular that $s$ is an increasing sequence, that is $i_0<i_1<\ldots<i_{k}$.  
Recall that $\widetilde{S}_k\subset S_k$ is the subgroup of permutations $\sigma$ such that $\sigma(k-1)<\sigma(k)$. 
For any $\sigma\in \widetilde{S}_{k}$, let $H_\sigma(s)$ and $H^{-1}_\sigma(s)$ be the $C_k$-trees for $\1$ represented on the left-hand side of Figure \ref{fig:Hsigma} (see Convention \ref{conv}).
We stress that $H^{\pm}_\sigma(s)$ intersects the $i$th component of $\1$ if and only if $i=i_j$ for some $j$.
 \begin{figure}[!h]
  \includegraphics[scale=1]{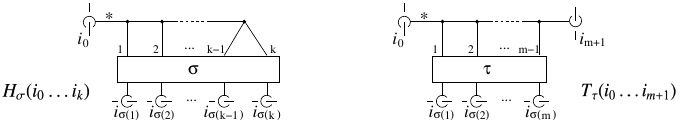}
  \caption{The $C_k$-tree $H_\sigma(i_0 \ldots i_{k})$ (left) and the $C_{m+1}$-tree $T_\tau(i_0 \ldots i_{m+1})$ (right).}\label{fig:Hsigma}
 \end{figure} 
We also stress that $H^{\pm}_\sigma(s)$ is linear, and that its left end intersects the component of $\1$ with smallest index $i_0$. 

Similar to (\ref{eq:0}), we have the following consequence of Lemma \ref{lem:twist}: 
\begin{equation}\label{eq:0L}
\textrm{For all $s=i_0 i_1 \ldots i_{k}$ and all $\sigma\in \widetilde{S}_{k}$,\,\,\,\,\,\,\,} \1_{H_\sigma(s)}\cdot \1_{H^{-1}_\sigma(s)} \stackrel{C_{k+1}}{\sim} \1.
\end{equation} 

Fixing some total order on $\widetilde{S}_{k}$, denote by $H$ the split union of $C_k$-trees in the product $\prod_{\vert s\vert =k+1} \prod_{\sigma\in \widetilde{S}_{k}} \1_{H_\sigma(s)}\cdot\1_{H^{-1}_\sigma(s)}$, where the first product is over all length $k+1$ subsequences $s$ of $12...(n-1)$, ordered lexicographically. 
We stress that $H$ is disjoint from the $n$th component of $\1$. 

Now let $L$ be a string link with vanishing Milnor link-homotopy invariants up to length $\le k$.  
Let us prove by induction that for all $m\ge k$ we have 
$$ L\stackrel{C_{m+1}}{\sim} \1_{H\cup E_m\cup R_m}, $$
where $E_m$ is a union of $H$-trivial tree claspers for $\1$, and where $R_m$ is a union of repeated claspers. 
The result will follow, since $H$ consists of exactly $\binom{n-1}{k+1}k!$ $C_k$-trees and since both $C_{m+1}$-equivalence and surgery along repeated claspers imply link-homotopy for $m\ge n$ (Lemma  \ref{lem:lhclasper}). 

The case $m=k$ is given by (\ref{eq:0L}), as follows. By Lemma \ref{rem:Cm+lh}, we have that $L$ is $C_{k+1}$-homotopic to $\1$, which by Lemma \ref{lem:lhclasper} and \cite[Prop.~3.22]{Habiro} implies that $L\stackrel{C_{k+1}}{\sim} \1_{R_1}$ for some union $R_1$ of repeated claspers. Equation (\ref{eq:0L}) then gives that $ L\stackrel{C_{k+1}}{\sim} \1_{H\cup R_1}$.

Now suppose that $L\stackrel{C_{m+1}}{\sim} \1_{H\cup E_m\cup R_m}$ for some union $E_m$ of $H$-trivial claspers and some union of repeated claspers $R_m$. 
Then there exists a union $F_{m+1}$ of simple $C_{m+1}$-trees for $\1$ such that 
 $$
 L = \1_{H\cup E_m\cup R_m\cup F_{m+1}}.
 $$
Denote by $F^r_{m+1}$ the subset of $F_{m+1}$ consisting of all repeated claspers, and set $R_{m+1}=R_m\cup  F^r_{m+1}$. 
Then by Lemma \ref{lem:calculus}, Corollary \ref{cor:asihx} and the Twist Lemma \ref{lem:twist}, we have  
 $$ L \stackrel{C_{m+2}}{\sim} \1_{H\cup E_m\cup R_{m+1}}\cdot P_{m+1}, $$
where $P_{m+1}$ is a product, each factor being obtained from $\1$ by surgery along a copy of the $C_{m+1}$-tree $T_\tau(s)$ or $T^{- 1}_\tau(s)$ shown in Figure \ref{fig:Hsigma} for some length $m+2$ subsequence $s=i_0i_1\ldots i_{m+1}$ of $12\ldots n$ and some $\tau\in S_{m}$. 
Note that $T^{\pm}_\tau(s)$ is linear, and that its left end (resp. right end) intersects the component of $\1$ with smallest index $i_{0}$ (resp. with largest index $i_{m+1}$). 

Without loss of generality, we can focus on a factor $\1_{T_\tau}(s)$ for some subsequence $s=i_0i_1\ldots i_{p+1}$ of $12\ldots n$ and some $\tau\in S_m$. 
Consider the length $k+1$ subsequence $\ov{s}=i_0 i_1\ldots i_k$ of $s$, 
and denote by $\overline{\tau}\in S_k$ the restriction of $\tau$ to $\{1,\ldots,k\}$. 
If $\overline{\tau}\in \widetilde{S}_k$, then $H_{\overline{\tau}}(\ov{s})$ is one of the $\binom{n-1}{k+1}k!$ $C_k$-trees in $H$: using Lemma \ref{lem:calculus} we can deform $H_\tau(s)$ up to $C_{m+2}$-equivalence so that, in a neighborhood of $H_{\overline{\tau}}(\ov{s})$, the union $H_{\overline{\tau}}(\ov{s})\cup H_\tau(s)$ is exactly as shown on the left-hand side of Figure \ref{fig:tech}. 
We can then apply Lemma \ref{lem:tech}, and we obtain in this way that $\1_{ H_{\overline{\tau}}(\ov{s})\cup H_\tau(s)}\stackrel{C_{m+2}}{\sim} \1_{H_{\overline{\tau}}(\ov{s})\cup C}$, where $C$ is an $H$-trivial $C_{m-k}$-tree.  
If $\overline{\tau}\in S_k\setminus \widetilde{S}_k$, then we can proceed exactly as in Section \ref{sec:knot}, using Remark \ref{rem:tech} instead of Lemma \ref{lem:tech}. 

Applying this procedure to all factors in $P_{m+1}$ successively, yields the desired equivalence. 
This concludes the induction, and the proof of Theorem \ref{thm:linklh}.

\begin{rem}\label{rem:k=1link}
The proof of Theorem \ref{thm:linklh} can be easily adapted to the case $k=1$, as follows.\\[-0.3cm] 

\noindent \begin{minipage}{9.5cm}
For each pair $(i,j)$ of integers in $\{1,\cdots,n-1\}$ such that $i<j$, let $C_{ij}$ denote the $C_1$-tree for $\1$ shown to the right. Set 
$$ H:=\prod_{i<j} \1_{C_{ij}}\cdot  \1_{C_{ij}^{-1}}, $$
where the product is according to the lexicographic order. 
\end{minipage}
\begin{minipage}{3.5cm}
\,\,\,\,\,\, \includegraphics[scale=1]{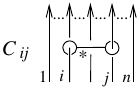}\\[-0.45cm]
\end{minipage}
Note that the number of $C_1$-trees in $H$ is $2\binom{n-1}{2}=(n-1)(n-2)$. 
From then on, the same inductive argument as above applies. 
(Note that the discussion is in fact simpler, since for $k=1$ we no longer have to distinguish the cases where $\overline{\tau}\in \widetilde{S}_k$ and $\overline{\tau}\in S_k\setminus \widetilde{S}_k$, and we can always make use of Lemma \ref{lem:tech}.)
\end{rem}

\subsection{$\Delta$-unknotting number and concordance: proof of Theorem \ref{thm:linkconc}}\label{sec:proofconc}

The proof shall follow the same global strategy as for Theorems \ref{thm:knots} and \ref{thm:linklh}. 

Let $D$ be union of $2\binom{n}{2}+2\binom{n}{3}=\frac{1}{3}(n^3-n)$ simple $C_2$-trees for $\1$ given by
$$ \1_D = \prod_{1\le i<j\le n} (\1_{W_{ij}}\cdot \1_{W'_{ij}})  \prod_{1\le i<j<k \le n} (\1_{B_{ijk}}\cdot \1_{B^{-1}_{ijk}}), $$
where $W_{ij}$, $W'_{ij}$ and $B^{\pm 1}_{ijk}$ are represented in Figure \ref{fig:C2}. 
   \begin{figure}[!h]
  \includegraphics[scale=1]{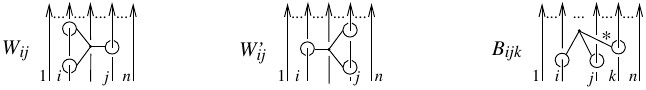}
  \caption{The $C_2$-trees $W_{ij}$, $W'_{ij}$ and $B_{ijk}$. }\label{fig:C2}
 \end{figure} 

Note that, for all $i,j$ with $i<j$, we have $\1_{W_{ij}}\simeq \1_{W'_{ij}}$: this is merely a consequence of the symmetry of the Whitehead (string) link, see \cite[Fig.~6]{jbjktr}. Combining this fact with the AS and STU relations (Lemma \ref{lem:asihx}), we have  
 \begin{equation}\label{eq:conc0}
    \textrm{For all $i,j$ with $i<j$,} \quad \1_{W_{ij}}\cdot \1_{W'_{ij}}\stackrel{C_3}{\sim} 1_{G_{ij}},
 \end{equation}
 where $G_{ij}$ is some looped clasper. 

Let $L$ be a $C_3$-trivial string link. 
We will prove by induction that for all $p\ge 2$ we have 
$$ L\stackrel{C_{p+1}}{\sim} \1_{D\cup E_p\cup G_p}, $$
where $E_p$ is a union of $D$-trivial tree claspers for $\1$, and where $G_p$ is a union of looped claspers. 
This will imply the desired result with $J=\1_{D\cup E_p}$. Indeed, on one hand since $D$ consists of $\frac{1}{3}(n^3-n)$ $C_2$-trees and $E_p$ is $D$-trivial, we have the desired upper bound on $u_2(J)$. On the other hand, by Lemma \ref{lem:loop} $L$ is $C_{p+1}$-concordant to $J$, hence cannot be distinguished from $J$ by any finite type concordance invariant of degree $\le p$. 

Consider first the case $p=2$. 
By assumption we have  $L\stackrel{C_{3}}{\sim} \1$.  For all $i,i$ ($i<j$), let $\overline{G_{ij}}$ be a looped clasper obtained by adding a twist on some edge of the clasper $G_{ij}$ given by Equation (\ref{eq:conc0}).
Then by Lemmas \ref{lem:calculus} and \ref{lem:twist} we have 
$$  \1_{D}\cdot \prod_{1\le i<j\le n} \1_{\overline{G_{ij}}}\stackrel{C_{3}}{\sim} \1\stackrel{C_{3}}{\sim} L,  $$
as desired.

Now suppose that $L\stackrel{C_{p+1}}{\sim} \1_{D\cup E_p\cup G_p}$ for some union $E_p$ of $D$-trivial claspers and some union of looped claspers $G_p$. 
By Lemma \ref{lem:calculus} and Corollary \ref{cor:asihx}, we have 
$$ 
L \stackrel{C_{p+2}}{\sim} \1_{D\cup E_p\cup G_p}\cdot T_{p+1},
$$
where $T_{p+1}$ is a product, each factor being obtained from $\1$ by surgery along a linear $C_{p+1}$-tree.
Pick a factor $\1_T$ in $T_{p+1}$, and consider one of its two ends, intersecting (say) the $i$th component of $\1$.  
Then by the Twist Lemma \ref{lem:twist}, one may assume up to $C_{p+2}$-equivalence that $\1\cup T$ intersects a $3$-ball in exactly one of the following configurations, where $i,j,k$ are pairwise distinct: 
$$ \textrm{\includegraphics[scale=1]{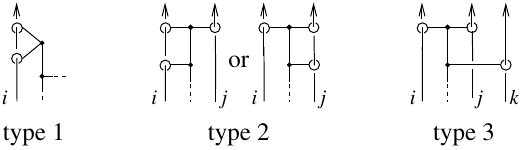}} $$
\begin{itemize}
\item If $T$ is of the first type, then Assertion (1) of Lemma \ref{cor:loop} implies that $\1_P\stackrel{C_{p+1}}{\sim} \1_{G_P}$ for some looped clasper $G_P$.
\item If $T$ is of the second type, then we can use Lemma \ref{lem:calculus} to deform $T$ so that $W_{ij}\cup T$ (resp. $W'_{ij}\cup T$) is exactly as on the left-hand side of Figure \ref{fig:tech} in a neighborhood of $W_{ij}$ (resp. $W_{ij}$). Lemma \ref{lem:tech} then implies that 
 $\1_{D\cup E_p\cup G_p}\cdot \1_T\stackrel{C_{p+2}}{\sim} \1_{D\cup E'_p\cup G_p}$, 
 where $E'_p$ is obtained from $E_p$ by adding one $D$-trivial $C_{p-1}$-tree. 
\item If $T$ is of the third type, the same argument applies, using the $C_2$-tree $B_{ijk}$ in $D$ instead of $W_{ij}$. 
\end{itemize}
Summarizing, for an arbitrary factor $\1_T$ of $T_{p+1}$ we have that 
 $\1_{D\cup E_p\cup G_p}\cdot \1_T\stackrel{C_{p+2}}{\sim} \1_{D\cup E'_p\cup G'_p}$, 
for some union $E'_p$ of $D$-trivial tree claspers and some union $G'_p$ of looped claspers. 
The conclusion then follows by applying this fact iteratively to all factors of $T_{p+1}$.
This completes the proof of Theorem \ref{thm:linkconc}

\begin{proof}[Proof of Corollary \ref{cor:linkconc}]
Any $C_3$-trivial link $\widetilde{L}$ is the closure of a $C_3$-trivial string link $L$: this is simply because $\widetilde{L}$ is obtained from the unlink by surgery along a union of (simple) $C_3$-trees (\cite[Prop.~3.22]{Habiro}), which can always be assumed to be disjoint from a disk meeting each link component once transversely and with same local sign. 
Now let $p$ be an integer. 
Following the proof of Theorem \ref{thm:linkconc} above, we have that 
$L$ is $C_{p+1}$-equivalent to a product 
$\1_{D\cup E_p\cup G_p}$, where $E_p$ is a union of $D$-trivial claspers and $G_p$ is a union of looped claspers. 
Hence there exists a union $F_{p+1}$ of $C_{p+1}$-trees for $\1$ such that   $L=\1_{D\cup E_p\cup G_p\cup F_{p+1}}$. 
Deleting $D$, which is a union of  $\frac{1}{3}(n^3-n)$ simple $C_2$-trees, gives the string link $\1_{G_p\cup F_{p+1}}$, which is $C_{p+1}$-concordant to $\1$ by Lemma \ref{lem:loop}. 
In other words, $L$ can be deformed into a string link with vanishing Milnor invariants of length $\le p$ by at most $\frac{1}{3}(n^3-n)$ $\Delta$-moves. 
The conclusion follows, since the closure of a string link with vanishing Milnor invariants of length $\le k$, has the same vanishing property. 
\end{proof}

\subsection{$C_k$-unknotting number and concordance}\label{sec:moreconc}
In this subsection we discuss several extensions of Theorem \ref{thm:linkconc}.
\medskip 

We first outline how the strategy used above can be applied to all higher order unknotting numbers. 
More precisely, suppose that a string link $L$ is $C_{k+1}$-trivial and that we seek an upper bound for the $C_k$-unknotting number of a link $J$ sharing all finite type concordance invariants with $L$ up to some given degree. 
Following the above proof, this bound is given by the cardinality of the clasper $C$, which provides a \lq reservoir' of tree claspers allowing us to apply Lemma \ref{lem:tech} to all possible higher degree tree claspers appearing during the induction. 
By Assertion (2) of Lemma \ref{cor:loop} and Corollary \ref{cor:asihx}, it suffices to consider all linear trees with $k+1$ univalent vertices labeled by $\{1,\ldots,n\}$: this argument shows that $u_k(J)\le n^{k+1}$. 

This is of course a very rough upper bound, and the proof of Theorem \ref{thm:linkconc} above illustrates how it can be refined by a closer analysis of the higher degree claspers involved in the induction process. We shall however not pursue this direction in this paper. 
Instead, we conclude this section with a refinement of the $\Delta$-unknotting number bound given in Theorem \ref{thm:linkconc}. 
\medskip

Since Theorem \ref{thm:linkconc} gives an upper bound for the $C_2$-unknotting number, it is natural to extend it to the case 
where $L$ is $C_2$-trivial, that is, $L$ is an algebraically split string link. 
In that case, following \cite[(5.1)]{MY5} we have  
\begin{equation}\label{eq:laté}
L\stackrel{C_3}{\sim}\prod_{1\leq i\leq n} T_i^{\, \textrm{Arf}(L_i)}\cdot \prod_{1\leq i< j\leq n} (\1_{W_{ij}})^{\, \beta^{(2)}(L_i\cup L_j)}\cdot \prod_{1\leq i<j<k\leq n} (\1_{B_{ijk}})^{\, \mu_L(ijk)}\cdot \1_{G},
\end{equation}
where $T_i$ is a copy of the (long) right-handed trefoil on the $i$th component, $W_{ij}$ and $B^{\pm 1}_{ijk}$ are given in Figure \ref{fig:C2} and $G$ is some union of looped claspers, and 
where Arf$(L_i)\in \{0,1\}$ is the Arf invariant of $L_i$, $\beta^{(2)}(L_i\cup L_j)\in \{0,1\}$ is the mod $2$ Saton-Levine invariant of $L_i\cup L_j$ and $\mu_L(ijk)\in \mathbb{Z}$ is Milnor's triple linking number of $L_i\cup L_j\cup L_k$.

The proof then proceed exactly as above, with the following modifications:
\begin{itemize}
\item Since the trefoil has $C_2$-unknotting number $1$, each nontrivial factor $T_i^{\, \textrm{Arf}(L_i)}$ in (\ref{eq:laté}) increases the $C_2$-unknotting number of $J$ by $1$ compared to the statement of Theorem \ref{thm:linkconc};
\item Each nontrivial factor $(\1_{W_{ij}})^{\, \beta^{(2)}(L_i\cup L_j)}$ in (\ref{eq:laté}), increases $u_2(J)$ by $1$, since $L$ contains a copy of the clasper $W_{ij}$; 
\item Each nontrivial factor $(\1_{B_{ijk}})^{\, \mu_L(ijk)}$ in (\ref{eq:laté}) increases $u_2(J)$ by $\vert \mu_L(ijk)\vert$, the number of copies of the clasper $B^{\pm 1}_{ijk}$ in (\ref{eq:laté}). But since any of these copies can be used to apply Lemma \ref{lem:tech} when necessary, we can delete the factor $(\1_{B_{ijk}}\cdot \1_{B^{-1}_{ijk}})$ in $C$, thus decreasing the $C_2$-unknotting number of $J$ by $2$. Summarizing, the $C_2$-unknotting number of $J$ changes compared to Theorem \ref{thm:linkconc} by 
$$ 
M(ijk):= \left\{ \begin{array}{cl} 
             0 & \textrm{if $\mu_L(ijk)=0$} \\
             \vert \mu_L(ijk)\vert-2 & \textrm{if $\mu_L(ijk)\neq 0$.}  \end{array}\right.            
$$
\end{itemize}
This shows the following.
\begin{thm}
Let $L$ be an algebraically split link and let $p$ be an integer. 
There exists a link $J$ with $\Delta$-unknotting number 
$$ u_2(J)\le \frac{1}{3}(n^3-n)+\sum_{1\le i\le n} \textrm{Arf}(L_i) + \sum_{1\le i<j\le n} \beta^{(2)}(L_i\cup L_j) 
+ \sum_{1\le i<j<k\le n} M(ijk) $$  
such that $v_p(K)=v_p(J)$ for all finite type concordance invariants $v_p$ of degree $\le p$. 
\end{thm}

\section{The case of Ribbon $2$-knots: proof of Theorem \ref{thm:2knots}}\label{sec:ribbon}

The goal of this section is to prove Theorem  \ref{thm:2knots} for ribbon $2$-knots. 
Before doing so in Section \ref{sec:proofribbon}, we recall some necessary material on the normalized Alexander polynomial, welded theory and arrow calculus. 

\subsection{The normalized Alexander polynomial}\label{sec:alex}
The Alexander polynomial $\Delta_R(t)\in \mathbb{Z}[t^{\pm 1}]$ of a ribbon $2$-knot $R$ is well-defined up to multiplication by $\pm t^{\pm n}$. 
By imposing that $\Delta_R(1)=1$ and $\frac{d\Delta_R}{dt}(1)=0$, we obtain the well-defined \emph{normalized Alexander polynomial} of $R$, denoted by $\tilde\Delta_R$ \cite{HKS}. 
Taking the power series expansion at $t=1$  
$$
\tilde\Delta_L(t) = 1 + \sum_{k\ge 2} \alpha_k(L) (1-t)^k 
$$
defines an infinite family of integer-valued invariants $\alpha_k$, called 
the $k$th \emph{normalized coefficient} of the Alexander polynomial (this differs from \cite[\S~2]{HKS} by a factor $(-1)^k$).

For each $k\ge 2$, $\alpha_k$ is a finite type invariant of degree $k$ in the sense of \cite{HKS}, and any finite type invariant of ribbon $2$-knots is a polynomial in these $\alpha_k$ by \cite{HS} (see also \cite[\S~10.2]{arrow}).

\subsection{Welded objects and the Tube map}\label{sec:Tube}

Ribbon $2$-knots are closely related to diagrammatic objects called welded (long) knots.

A \emph{welded long knot} is the equivalence class of a properly immersed oriented interval in the $2$-disk, whose transverse double points are decorated either as classical crossings (as in usual knot diagrams) or as virtual ones (represented by a circled double point); the welded equivalence  is generated by the usual Reidemeister moves, the virtual Reidemeister moves (Reidemeister moves where all crossings are replaced by virtual ones), 
 the Mixed and the OC moves shown below:\\[-0.5cm]
\[
    \begin{array}{cc}
      \dessin{1.5cm}{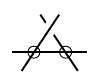}\stackrel{\textrm{Mixed}}{\longleftrightarrow} \dessin{1.5cm}{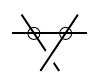} \, & \, 
      \dessin{1.5cm}{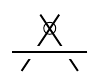}\ \stackrel{\textrm{OC}}{\longleftrightarrow}
      \dessin{1.5cm}{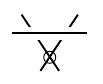}
    \end{array}.
\]

A basic local move on welded long knots is the \emph{virtualization move}, which replaces a classical crossing by a virtual one. It is easily seen that this is an unknotting operation: we can thus define the \emph{virtualization number} of a welded long knot, as the minimal number of virtualization moves needed to trivialize it. 
\medskip

There is a surjective \emph{Tube map} from welded long knots  to ribbon $2$-links, due to  Satoh in \cite{Satoh}, which can roughly be described as follows. 
Given a welded long knot, the Tube map assigns to each classical crossing a pair of annuli locally linked in a $4$-ball as shown in \cite[Fig.~6]{Satoh}. Next we connect these annuli to one another by unlinked annuli, as combinatorially prescribed by the diagram, so as to obtain a single knotted annulus. It only remains to glue a $2$-disk at each boundary component of this annulus (corresponding to the two boundary points of the welded long knot): the result is the desired (diagram of a) $2$-knot. 

As noted in \cite{ABMW}, two welded long knots which differ by a virtualization move, have images by the Tube map that differ by a ribbon crossing change. Hence the virtualization number of a welded long knot $D$ gives an upper bound for the ribbon unknotting number of Tube$(D)$.
\medskip

Welded (long) knots admit a notion of \lq fundamental group', simply defined by the usual Wirtinger algorithm with no extra relation at virtual crossings. Any classical invariants extracted from the knot group thus has a natural welded generalization. There is in particular a normalized Alexander polynomial for welded long knots, from which one extracts normalized coefficients $\alpha_k$ exactly as in Section \ref{sec:alex}, see \cite[\S~6.2]{arrow}. 
Moreover, the Tube map is well-behaved with respect to these constructions. In particular we have: 
\begin{fact}
Let $D$ be a welded long knot. Then for all $k\ge 2$ we have $\alpha_k\left( \textrm{Tube}(D)\right) = \alpha_k(D)$. 
\end{fact}
 
\subsection{Arrow calculus}

There is a welded analogue of clasper calculus for welded objects, called \emph{arrow calculus} \cite{arrow}. 
We give a brief overview of this theory, mentioning only the ingredients that will be useful for our purpose. The reader is referred to \cite{arrow} for  further details.
\medskip

Let $D$ be a welded diagram. 
A \emph{w-arrow} for $D$ is an oriented interval, possibly decorated by a $\bullet$, immersed in the plane so that the endpoints lie on $D$. \emph{Surgery} on $D$ along a w-arrow introduces a classical and a virtual crossing as follows:
\[
\dessin{1.5cm}{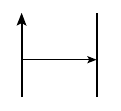}\leadsto\dessin{1.5cm}{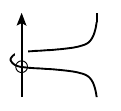}
\hspace{1cm}
\dessin{1.5cm}{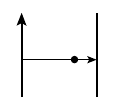}\leadsto\dessin{1.5cm}{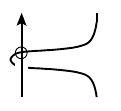}.
\]

More generally, for any $k\ge 1$ a \emph{w-tree of degree $k$} for $D$ is an immersed oriented unitrivalent tree with $k-1$ trivalent vertices,  whose endpoints (univalent vertices) lie on $D$; the orientation is such that each trivalent vertex has one outgoing and two ingoing edges. 
Surgery along a w-tree is defined recursively by the Expansion rules (E):
\[
\dessin{1.5cm}{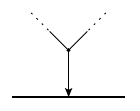}\stackrel{(E)}{=}\dessin{1.5cm}{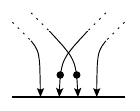}
\hspace{1cm}
\dessin{1.5cm}{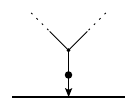}\stackrel{(E)}{=}\dessin{1.5cm}{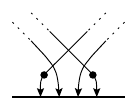}.
\]
In other words, a w-tree of degree $k$ encodes a $k$-fold commutator of w-arrows. A w-tree of degree $1$ is just a w-arrow. 
\medskip

Any welded diagram can be represented by the trivial diagram $\1$ together with a union of w-arrows or w-trees.  
Much like Habiro's clasper calculus, there is a collection of moves on w-arrows and w-trees which relate two such presentations having equivalent surgery results. 
For the purpose of this paper, we shall only need the following, see \cite[Lem.~5.9]{arrow}:
\begin{lem}[Inverse]\label{lem:inversea}
Two parallel w-trees of degree $k$ ($k\ge 1$) which only differ by a decoration $\bullet$ as shown below, yield a trivial surgery.
\begin{center}
   $\quad \quad$ \includegraphics[scale=0.8]{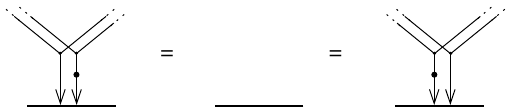}
\end{center}
\end{lem}

 For each $k\ge 1$, the $w_k$-equivalence is the equivalence relation on welded long knots denoted by $\stackrel{w_{k}}{\sim}$ and  generated by welded equivalence and surgery along w-trees of degree $l$ for $l\ge k$. 
This is analogous to the $C_k$-equivalence of clasper theory, and allows likewise for further flexibility in manipulating w-trees. 
In particular we have the following, see \cite[Cor.~7.13]{arrow}: 
\begin{lem}\label{lem:exchange}
Let $T$ and $T'$ be w-trees of degree $k$ and $k'$ for $\1$ ($k,k'\ge 1$). 
We can freely exchange the relative position of two adjacent endpoints of $T$ and $T'$, at the cost of adding w-trees of degree $\ge k+k'$. 
\end{lem}

Finally, recall that the normalized coefficient $\alpha_k$ of welded long knots ($k\ge 2$) is a $w_{k+1}$-equivalence invariants \cite{arrow}. 
\subsection{Proof of Theorem \ref{thm:2knots}}\label{sec:proofribbon}

By Section \ref{sec:Tube},  Theorem \ref{thm:2knots} is in fact a corollary of the following. 
\begin{thm}\label{thm:weldedknots}
Let $K$ be a welded long knot, and let $n$ be an integer ($n\ge 2$). 
There exists a welded long knot $J$ with virtualization number $1$ and such that $\alpha_k(K)=\alpha_k(J)$ for all $k\le n$. 
\end{thm}

\begin{proof}
Consider the w-arrow $A$ shown on the left-hand side of Figure \ref{fig:A}.  
 \begin{figure}[!h]
  \includegraphics[scale=0.75]{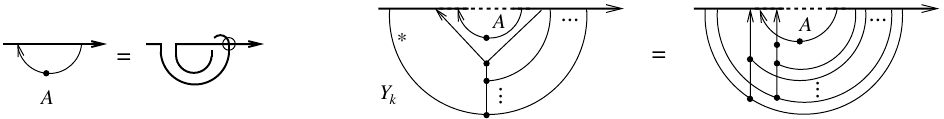}
  \caption{The w-arrow $A$ (left), and the deletion process for $Y_{p+1}$ (right).}\label{fig:A}
 \end{figure}
 
We  prove by induction that for all $k\ge 1$ we have 
$$ K\stackrel{w_{n+1}}{\sim} \1_{A\cup C_k}\cdot \1_{W_k}, $$
where $W_k$ is a union of w-trees of degree $\ge k$ and $C_k$ is an $A$-trivial union of w-trees, that is, a union of w-trees such that $\1_{C_k}=\1$. 

The case $k=1$ is clear. As Figure \ref{fig:A} illustrates, we have $\1_{A}=\1$, and since the welded (long) knot $K$ is obtained by surgery on $\1$ along a union $W_1$ of w-arrows, see \cite[Prop.~4.2]{arrow} we have  $K=\1_{A}\cdot \1_{W_1}$.
 
Now assume that $K\stackrel{w_{n+1}}{\sim} \1_{A\cup C_k}\cdot \1_{W_k}$ for some $A$-trivial union $C_k$ of w-trees and some  union $W_k$ of w-trees of degree $\ge k$. 
By \cite[Lem.~8.9]{arrow}, we have 
$\1_{W_k}\stackrel{w_{n+1}}{\sim} (Y_{k})^x\cdot \1_{W_{k+1}}$ where $Y_{k}$ is the $w_{k}$-tree illustrated in Figure \ref{fig:A},\footnote{More precisely, the w-tree $Y_{k}$ used here differs from $L_k$ in  \cite[Lem.~8.9]{arrow} only by a number of $\bullet$; hence by \cite[Lem.~7.7]{arrow} and the involutivity of the $\bullet$ decorations we have $\1_{Y_{k}}\stackrel{w_{k+1}}{\sim} (\1_{L_k})^{(-1)^k}$.} $x\in \mathbb{Z}$ and $W_{k+1}$ is some union of w-trees of degree $\ge k+1$. 
Therefore 
\begin{equation}\label{eq:w}
 K\stackrel{w_{n+1}}{\sim} \1_{A\cup C_k}\cdot (Y_{k})^x\cdot \1_{W_{k+1}}. 
 \end{equation}
Suppose $x>0$ (the case $x<0$ is strictly similar). Using Lemma \ref{lem:exchange}, we can deform a copy of $Y_{k}$ in (\ref{eq:w})  so that $Y_{k}\cup A$ is exactly as show on the right-hand side of Figure \ref{fig:A}. Applying an Expansion move (E) gives the equality shown in the figure, deleting $Y_{k}$ at the expense of a pair of w-trees which is $A$-trivial by Lemma \ref{lem:inversea}. 
Applying this process repeatedly to all factors $Y_{k}$ in (\ref{eq:w}), we obtain that 
$$ K\stackrel{w_{n+1}}{\sim} \1_{A\cup C_{k+1}}\cdot \1_{W_{k+1}},  $$
as desired, where $C_{k+1}$ denotes the union of $C$ and $\vert x\vert$ pairs of $A$-trivial w-trees. 
This concludes the induction. 

It remains to observe that the claim we just proved by induction implies Theorem \ref{thm:weldedknots}.

On one hand, deleting $A$ gives the welded long knot $\1_{C_n}=\1$, so that the welded long knot $J_n:=\1_{A\cup C_n}$ satisfies $v(J_p)=1$.
On the other hand, we have $J_n\stackrel{w_{n+1}}{\sim} K$, and the result follows since the Alexander coefficients $\alpha_k$ with $k\le n$ are $w_{n}$-equivalence invariants.
\end{proof}

\appendix
\section{Proof of Lemma \ref{lem:tech}}\label{sec:A}

This appendix provides the proof of the technical Lemma \ref{lem:tech}. Although the argument essentially boils down to Habiro's Zip-construction \cite[\S~3.3]{Habiro}, the statement we aim at here is more precise.

In \cite[Prop.~2.7]{Habiro}, Habiro gives a list of 12 moves relating claspers with equivalent surgery results.   
In this appendix we will freely use these \emph{Habiro moves}, which are essentially derived from Kirby calculus, by referring to their numbering in \cite{Habiro}. These moves in particular involve an additional ingredient for claspers, namely \emph{boxes}.  
A box of a clasper is a disk with three edges attached but, unlike for nodes, one incident edge is distinguished from the other two; this distinction is done by drawing boxes as rectangles, see \cite[Fig.~5 and 6]{Habiro}.
\medskip

We can now proceed with the proof of Lemma \ref{lem:tech}. 
Figure \ref{fig:prooftech} summarizes each step of the proof in the case $k=3$ and $m=5$. 
 \begin{figure}[!h]
  \includegraphics[scale=1]{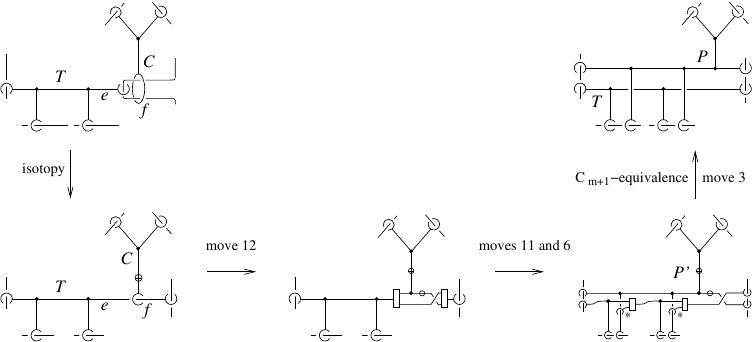}
  \caption{Proof of Lemma \ref{lem:tech}.}\label{fig:prooftech}
 \end{figure}
 
Consider the union $T\cup C$ as in the top left of Figure \ref{fig:prooftech}, and isotope the leaf $f$ of $C$ so that it is disjoint from $\1$, but intersect instead the edge $e$ of $T$ at a single point (bottom left of the figure). This configuration allows us to apply Habiro's move 12 at this leaf, introducing two boxes (bottom center). 
We can now \lq push' iteratively these boxes across nodes and leaves by using moves 12 and 6, respectively, so as to obtain a $C_m$-tree $P'$, which differs from $P$ only by a pair of half-twists, and whose edges may be linked by small leaves, namely the $\ast$-marked leaves shown in the bottom right of the figure. 
Now, using Lemma \ref{lem:calculus} we can unlink the edges from these $\ast$-leaves up to $C_{m+1}$-equivalence; the $\ast$-leaves being now trivial, we can apply move 3 to replace all boxes by edges. Furthermore, applying the Twist Lemma \ref{lem:twist} allows us to remove the two half-twists on $P'$  up to $C_{m+1}$-equivalence. The result is precisely the union $T\cup P$ (top right of the figure), which completes the proof.
%

\bibliographystyle{abbrv}
\bibliography{References}

\end{document}